\documentclass{article}%
\usepackage{amsmath}
\usepackage[nomarginpar]{geometry}
\usepackage{graphicx}
\usepackage{amsfonts}
\usepackage{amssymb}%
\providecommand{\U}[1]{\protect\rule{.1in}{.1in}}
\newtheorem{theorem}{Theorem}

\newtheorem{conjecture}[theorem]{Conjecture}
\newtheorem{corollary}[theorem]{Corollary}

\newtheorem{definition}[theorem]{Definition}

\newtheorem{lemma}[theorem]{Lemma}

\newenvironment{proof}[1][Proof]{\noindent\textbf{#1.} }{\ \rule{0.5em}{0.5em}}
\graphicspath{{D:/Dropbox/HararyWeight/Slike/}}
\begin{document}

\title{On inverse Wiener interval problem of trees}
\author{Jelena Sedlar\\University of Split, Faculty of civil engineering, architecture and geodesy, \\Matice hrvatske 15, HR-21000, Split, Croatia.}
\maketitle

\begin{abstract}
The Wiener index $W(G)$ of a simple connected graph $G$ is defined as the sum
of distances over all pairs of vertices in a graph. We denote by
$W[\mathcal{T}_{n}]$ the set of all values of Wiener index for a graph from
the class $\mathcal{T}_{n}$ of trees on $n$ vertices. The largest interval of
contiguous integers (contiguous even integers in case of odd $n$) contained in
$W[\mathcal{T}_{n}]$ is denoted by $W^{int}[\mathcal{T}_{n}].$ In this paper
we prove that both sets are of cardinality $\frac{1}{6}n^{3}+O(n^{2})$ in the
case of even $n,$ while in the case of odd $n$ we prove that the cardinality
of both sets equals $\frac{1}{12}n^{3}+O(n^{2})$ solving thus two conjectures
posed in literature.

\end{abstract}

%


{\bf AMS Subject Classifcation:} 05C35%

\section{Introduction}

The Wiener index of a connected graph $G$ is defined as the sum of distances
over all pairs of vertices, i.e.%
\[
W(G)=\sum_{u,v\in V(G)}d(u,v).
\]
It was first introduced in \cite{Wiener(1946)} and it was used for predicting
the boiling points of paraffins. Since the index was very succesful many other
topological indices were introduced later which use distance matrix of a
graph. There is a recent survey by Gutman et al. \cite{GutmanSurvey(2014)} in
which finding extremal values and extremal graphs for the Wiener index and
several of it variations is nicely presented. Given the class of all simple
connected graphs on $n$ vertices it is easy to establish extremal graphs for
the wiener index, those are complete graph $K_{n}$ and path $P_{n}$. The same
holds for the class of tree graphs on $n$ vertices in which minimal tree is
the star $S_{n}$ and the maximal tree is path $P_{n}.$ Many other bounds on
Wiener index are also established in literature.

In \cite{GutmanYeh(1995)} Gutman and Yeh proposed the inverse Wiener index
problem, i.e. for a given value $w$ the problem of finding a graph (or a tree)
$G$ for which $W(G)=w.$ The first attempt at solving the problem was made in
\cite{LepovicGutman(1998)} where integers up to $1206$ were checked and $49$
integers were found that are not Wiener indices of trees. In
\cite{BanBeregMustafa(2004)} it was computationally proved that for all
integers $w$ on the interval from $10^{3}$ to $10^{8}$ there exists a tree
with Wiener index $w$. The problem was finally fully solved in 2006 when two
papers were published solving the problem independently. It was proved in
\cite{Wang(2006)} that for every integer $w>108$ there is a caterpillar tree
$G$ such that $W(G)=w$. The other proof is from the paper \cite{Wagner(2006)}
where it was proved that all integers except those 49 are Wiener indices of
trees with diameter at most $4$.

A related question is to ask what value of the Wiener index can a graph (or a
tree) $G$ on $n$ vertices have? In order to clarify further this problem one
may also ask how many such values are there, how are they distributed along
the related interval or how many of them are contiguous. In
\cite{KrncSkrekovski(2016)} this problem is named the Wiener inverse interval
problem (see also a nice recent survey \cite{KnorSkrekovskiTepeh(2016)} which
covers the topic). In that paper the set $W[\mathcal{G}_{n}]$ is defined as
the set of all values of Wiener index for graphs $G\in\mathcal{G}_{n},$ where
$\mathcal{G}_{n}$ is a class of simple connected graphs on $n$ vertices.
Similarly, $W[\mathcal{T}_{n}]$ is defined as the set of values $W(G)$ for all
trees on $n$ vertices ($\mathcal{T}_{n}$ denotes the class of trees on $n$
vertices). Also, $W^{int}[\mathcal{G}_{n}]$ (or analogously $W^{int}%
[\mathcal{T}_{n}]$) is defined to be the largest interval of contiguous
integers such that $W^{int}[\mathcal{G}_{n}]\subseteq W[\mathcal{G}_{n}]$ (or
analogously $W^{int}[\mathcal{T}_{n}]\subseteq W[\mathcal{T}_{n}]$). In
\cite{KrncSkrekovski(2016)} the Wiener inverse interval problem on class
$\mathcal{G}_{n}$ was considered, while for the same problem on $\mathcal{T}%
_{n}$ following two conjectures were made.

\begin{conjecture}
\label{Con1}The cardinality of $W[\mathcal{T}_{n}]$ equals $\frac{1}{6}%
n^{3}+O(n^{2})$.
\end{conjecture}

\begin{conjecture}
\label{Con2}The cardinality of $W^{int}[\mathcal{T}_{n}]$ equals $O(n^{3})$.
\end{conjecture}

In this paper we will consider these two conjectures. First, we will prove
that for a tree $G$ on odd number of vertices vertices $n$ the value $W(G)$
can be only even number. That means that the inverse Wiener interval problem
in that case has to be reformulated as the problem of finding the largest
interval of contiguous even integers such that $W^{int}[\mathcal{T}%
_{n}]\subseteq W[\mathcal{T}_{n}].$ Since $\left\vert W[\mathcal{T}%
_{n}]\right\vert \leq W(P_{n})-W(S_{n})=\frac{1}{6}n^{3}-n^{2}+\frac{11}%
{6}n-1,$ we now conclude that the cardinality of $W[\mathcal{T}_{n}]$ in the
case of odd $n$ can be at most $\frac{1}{12}n^{3}+O(n^{2}).$ Given that
reformulation, we will prove both conjectures to be true. Even more, we will
prove the strongest possible version of Conjecture \ref{Con2} by proving that
$\left\vert W^{int}[\mathcal{T}_{n}]\right\vert $ also equals $\frac{1}%
{6}n^{3}+O(n^{2})$ (i.e. $\frac{1}{12}n^{3}+O(n^{2})$ in case of odd $n$)
which is the best possible result given the upper bound on $\left\vert
W[\mathcal{T}_{n}]\right\vert $ derived from $W(P_{n})-W(S_{n}).$

The present paper is organised as follows. It the next section basic
definitions and preliminary results are given. In the third section the
problem is solved for trees on even number of vertices, while in the fourth
section the problem is solved for trees on odd number of vertices.

\section{Preliminaries}

Let $G=(V(G),E(G))$ be a simple connected graph having $n=\left\vert
V(G)\right\vert $ vertices and $m=\left\vert E(G)\right\vert $ edges. For a
pair of vertices $u,v\in V(G)$ we define the distance $d_{G}(u,v)$ as the
length of the shortest path connecting $u$ and $v$ in $G.$ For a vertex $u\in
V$ the degree $\delta_{G}(u)$ is defined as the number of neighbors of vertex
$u$ in graph $G.$ When it doesn't lead to confusion we will use abbreviated
notation $d(u,v)$ and $\delta(u).$ Also, for a vertex $u\in V(G)$ and a set of
vertices $A\subseteq V(G)$ we will denote $d(u,A)=\sum_{v\in A}d(u,v).$ We say
that a vertex $u\in V(G)$ is a leaf if $\delta_{G}(u)=1,$ othervise we will
say that $u$ is interior vertex of a graph $G.$ We say that a vertex $u\in
V(G)$ is a petal if it has a leaf for a neighbor. A graph $G$ which does not
contain cycles is called a tree. We say that a tree $G$ is a caterpillar tree
if all its interior vertices induce a path. Such path will be called interior
path of a caterpillar. Let $a$ and $b$ be positive integers such that $a\leq
b.$ We say that the interval $[a,b]$ is Wiener $p-$complete if there is a tree
$G$ in $\mathcal{T}_{n}$ such that $W(G)=a+pi$ for every $i=0,\ldots
,\left\lfloor \frac{b-a}{p}\right\rfloor .$ We say that the interval $[a,b]$
is Wiener complete if it is Wiener $1-$complete.

Let us now prove that the value of the Wiener index for a tree on odd number
of vertices is even number.

\begin{theorem}
\label{tm_WienerEven}Let $G$ be a tree on odd number of vertices $n\geq3.$
Then $W(G)$ is even number.
\end{theorem}

\begin{proof}
The proof is by induction on $n.$ The only tree on $n=3$ vertices is $P_{3}$
for which it holds that $W(P_{3})=4$ which is even number. Let $G$ be a tree
on $n>3$ vertices where $n$ is an odd integer. Note that $G$ has at least two
leafs $u_{1}$ and $u_{2}$ which are neighboring to petals $v_{1}$ and $v_{2}$
(it can happen $v_{1}=v_{2}).$ Let $G^{\prime}$ be a tree on $n-2$ vertices
obtained by deleting leafs $u_{1}$ and $u_{2}.$ By induction hypotesis we know
that $W(G^{\prime})$ is even number. Now, note that%
\[
W(G)=W(G^{\prime})+d(u_{1},V^{\prime})+d(u_{2},V^{\prime})+d(u_{1},u_{2})
\]
where $V^{\prime}=V(G^{\prime}).$ Let $P$ be the path connecting vertices
$v_{1}$ and $v_{2}$ in $G^{\prime}.$ For any vertex $v\in V^{\prime}$ let $w$
be the vertex on $P$ which is closest to $v$ (i.e. $d(v,w)=\min\{d(v,x):x\in
V(P)\}).$ Note that $d(v,v_{1})+d(v,v_{2})=2d(v,w)+d(w,v_{1})+d(w,v_{2}%
)=2d(v,w)+d(v_{1},v_{2}).$ We now distinguish two cases.

CASE I. Suppose $d(v_{1},v_{2})$ is even. In that case $d(v,v_{1})+d(v,v_{2})$
is even for every $v\in V^{\prime}$ which implies $d(v,u_{1})+d(v,u_{2}%
)=d(v,v_{1})+1+d(v,v_{2})+1$ is also even for every $v\in V^{\prime}.$
Therefore $d(u_{1},V^{\prime})+d(u_{2},V^{\prime})$ must also be even. Further
note that $d(u_{1},u_{2})=d(v_{1},v_{2})+2$ is also even in this case. Since
$W(G^{\prime})$ is even by induction hypothesis, we conclude that $W(G)$ must
be even too.

CASE\ II. Suppose $d(v_{1},v_{2})$ is odd. Then $d(v,v_{1})+d(v,v_{2})$ is odd
for every $v\in V^{\prime}$ which further implies $d(v,u_{1})+d(v,u_{2})$ is
also odd for every $v\in V^{\prime}$. Since there is odd number of vertices in
$V^{\prime},$ we conclude that $d(u_{1},V^{\prime})+d(u_{2},V^{\prime})$ must
also be odd number. Also, note that $d(u_{1},u_{2})=d(v_{1},v_{2})+2$ is also
odd because $d(v_{1},v_{2})$ is odd. Therefore, $d(u_{1},V^{\prime}%
)+d(u_{2},V^{\prime})+d(u_{1},u_{2})$ is a sum of two odd numbers and
therefore must be even. Since $W(G^{\prime})$ is even, we now conclude that
$W(G)$ must be even.
\end{proof}

The main toool for obtaining our results throughout the paper will be a
transformation of a tree which increases the value of Wiener index by exactly
four. We will call it Transformation A, but let us introduce its formal definition.

\begin{definition}
Let $G$ be a tree and $u\in V(G)$ a vertex of degree $4$ such that neighbors
$v_{1}$ and $v_{2}$ of $u$ are leafs, while neighbors $w_{1}$ and $w_{2}$ of
$u$ are not leafs. We say that a tree $G^{\prime}$ is obtained from $G$ by
Transformation A if $G^{\prime}$ is obtained from $G$ by deleting edges
$uv_{1}$ and $uv_{2},$ while adding edges $w_{1}v_{1}$ and $w_{2}v_{2}.$
\end{definition}

\begin{theorem}
\label{tm_TransformationA}Let $G$ be a tree and let $G^{\prime}$ be a tree
obtained from $G$ by Transformation A. Then $W(G^{\prime})=W(G)+4.$
\end{theorem}

\begin{proof}
For the simplicity sake we will use notation $d^{\prime}(u,v)$ for
$d_{G^{\prime}}(u,v).$ Let $G_{w_{i}}=(V_{w_{i}},E_{w_{i}})$ be the connected
component of $G\backslash\{u\}$ which contains vertex $w_{i}$ for $i=1,2.$
Note that the only distances that change in Transformation A are distances
from vertices $v_{1}$ and $v_{2}.$ For every $v\in V_{w_{1}}\cup V_{w_{1}}$ we
have
\[
d^{\prime}(v_{1},v)-d(v_{1},v)+d^{\prime}(v_{2},v)-d(v_{2},v)=0.
\]
For the vertex $u$ we have
\[
d^{\prime}(v_{1},u)-d(v_{1},u)+d^{\prime}(v_{2},u)-d(v_{2},u)=2.
\]
Finally, we also have $d^{\prime}(v_{1},v_{2})-d(v_{1},v_{2})=2.$ Therefore,
$W(G^{\prime})-W(G)=4$ which proves the theorem.
\end{proof}

Although Transformation A can be applied on any tree graph, we will mainly
apply it on caterpillar trees. Moreover, it is critical to find a kind of
caterpillar tree on which Transformation A can be applied repeatedly as many
times as possible. For that purpose, let us prove the following theorem.

\begin{theorem}
\label{Tm_CaterpillarWienerComplete}Let $G$ be a caterpillar tree and
$P=u_{1}\ldots u_{d}$ its interior path. If there is a vertex $u_{i}\in P$ of
degree $4$ such that $u_{i\pm j}$ is of degree $3$ for every $j=1,\ldots,k-1,$
than the interval $[W(G),W(G)+4k^{2}]$ is Wiener $4-$complete.
\end{theorem}

\begin{proof}
By applying repeatedly Transformation A on exactly one vertex from $\{u_{i\pm
j}:j=0,\ldots,k-1\}$ until there is no more vertex of degree $4$ in that set,
one will make $k^{2}$ transformations and in each transformation the value of
Wiener index will increase by 4.
\end{proof}

Note that the transformation in Theorem \ref{Tm_CaterpillarWienerComplete} can
be applied $k^{2}$ times. If we prove that there are $O(n)$ different values
of $d$ for which $k=O(n),$ we obtain roughly $O(n^{3})$ graphs with different
values of Wiener index which is exactly the result we aim at. So, that is what
we are going to do in following sections, but in order to do that precisely we
will have to construct four different special types of caterpillar trees. To
easily construct those four types of caterpillar trees we first introduce two
basic types of caterpillars from which those four types will be constructed by
adding one or two vertices.

\begin{definition}
Let $n,d$ and $x$ be positive integers such that $n\geq18$ is even,
$\left\lceil \frac{n-2}{4}\right\rceil \leq d\leq\frac{n-8}{2}$ and
$x\leq\frac{4+4d-n}{2}.$ Caterpillar $B_{1}(n,d,x)$ is a caterpillar on even
number of vertices $n$ obtained from path $P=u_{-d}\ldots u_{-1}u_{0}%
u_{1}\ldots u_{d}$ by appending a leaf to vertices $u_{-d-1+x}$ and
$u_{d+1-x}$ and by appending a leaf to $2k-1$ consecutive vertices
$u_{-(k-1)},\ldots,u_{k-1}$ where $k=\frac{n-(2d+1)-1}{2}.$
\end{definition}

\begin{lemma}
\label{lema_WienerBasic1}Let $n,d$ and $x$ be integers such that
$B_{1}(n,d,x)$ is defined. Then
\begin{align*}
W(B_{1}(n,d,x))  &  =\frac{n^{3}}{3}+(-\frac{3d}{2}-\frac{5}{4})n^{2}%
+(4d^{2}+10d+\frac{13}{3}-2x)n+\\
&  +2x^{2}-\frac{8}{3}d^{3}-12d^{2}-\frac{46d}{3}-7.
\end{align*}

\end{lemma}

\begin{proof}
Note that for $k=\frac{n-(2d+1)-1}{2}$ and $x^{\prime}=-d-1+x$ we have
\begin{align*}
W(B_{1}(n,d,x))  &  =\sum_{i=-d}^{d}\sum_{j=i+1}^{d}(j-i)+\sum_{i=-(k-1)}%
^{k-1}\sum_{j=i+1}^{k-1}(j-i+2)+\\
&  +(2d+2-2(x-1))+\sum_{i=-d}^{d}\sum_{j=-(k-1)}^{k-1}(\left\vert
i-j\right\vert +1)+\\
&  +\sum_{i=-d}^{d}2(\left\vert i-x^{\prime}\right\vert +1)+\sum
_{i=-(k-1)}^{k-1}(2\left\vert i-x^{\prime}\right\vert +2).
\end{align*}
Simplifying this expresion yields the formula from the lemma statement.
\end{proof}

\begin{definition}
Let $n,d$ and $x$ be positive integers such that $n\geq18$ is even, $4\leq
d\leq\left\lfloor \frac{n}{4}\right\rfloor $ and $x\leq\frac{n-4d+2}{2}.$
Caterpillar $B_{2}(n,d,x)$ is a caterpillar on even number of vertices $n$
obtained from path $P=u_{-d}\ldots u_{-1}u_{0}u_{1}\ldots u_{d}$ by appending
a leaf to $2k-1$ consecutive vertices $u_{-(k-1)},\ldots,u_{k-1}$ where
$k=d-1,$ by appending $x$ leafs to each of the $u_{-(d-1)}$ and $u_{(d-1)},$
and by appending $r$ leafs to each of the $u_{-d}$ and $u_{d}$ where
$r=\frac{n-4d-2x+2}{2}.$
\end{definition}

\begin{lemma}
Let $n,d$ and $x$ be integers such that $B_{2}(n,d,x)$ is defined. Then
\[
W(B_{2}(n,d,x))=(\frac{d}{2}+1)n^{2}+(-2d-2)n-\frac{8d^{3}}{3}+\frac{32d}%
{3}-5+8x-8dx-2x^{2}.
\]

\end{lemma}

\begin{proof}
Let $k=d-1$ and $r=\frac{n-4d-2x+2}{2}.$ Note that%
\begin{align*}
W(B_{2}(n,d,x))  &  =\sum_{i=-d}^{d}\sum_{j=i+1}^{d}(j-i)+\sum_{i=-(k-1)}%
^{k-1}\sum_{j=i+1}^{k-1}(j-i+2)+4\tbinom{x}{2}+x^{2}(2d)+\\
&  +4\tbinom{r}{2}+r^{2}(2d+2)+\sum_{i=-d}^{d}\sum_{j=-(k-1)}^{k-1}(\left\vert
i-j\right\vert +1)+\\
&  +2x(3+\sum_{i=3}^{2d+1}(i-1))+2r\sum_{i=1}^{2d+1}i+2x\sum_{i=-(k-1)}%
^{k-1}(i+d+1)\\
&  +2r\sum_{i=-(k-1)}^{k-1}(i+d+2)+2(3xr+xr(2d+1)).
\end{align*}
Simplifying this expresion yields the desired formula.
\end{proof}

Finally, let us denote $d_{1}^{\min}=\left\lceil \frac{n-2}{4}\right\rceil $
and $x_{1}^{\max}=\frac{4+4d_{1}^{\min}-n}{2},$ while $d_{2}^{\max
}=\left\lfloor \frac{n}{4}\right\rfloor .$ Note that
\begin{equation}
B_{1}(n,d_{1}^{\min},x_{1}^{\max})=B_{2}(n,d_{2}^{\max},1).
\label{For_Lijepljenje}%
\end{equation}

\section{Even number of vertices}

\begin{definition}
Let $n,d$ and $x$ be integers for which $B_{1}(n-2,d,x)$ is defined. For
$s=-1,0,1,2$ caterpillar $G_{1}(n,d,x,s)$ is a caterpillar on even number of
vertices $n,$ obtained from $B_{1}(n-2,d,x)$ by appending a leaf to vertex
$u_{s}$ and a leaf to vertex $u_{d}$ of path $P=u_{-d}\ldots u_{-1}u_{0}%
u_{1}\ldots u_{d}$ in $B_{1}(n-2,d,x).$
\end{definition}

\begin{lemma}
\label{Lema_WienerGr1}Let $n,d$, $x$ and $s$ be integers for which
$G_{1}(n,d,x,s)$ is defined. Then%
\[
W(G_{1}(n,d,x,s))=W(B_{1}(n-2,d,x))+\frac{n^{2}}{4}+\frac{3n}{2}%
+2d^{2}+3d+2s^{2}-s-2x.
\]

\end{lemma}

\begin{proof}
Let $k=\frac{(n-2)-(2d+1)-1}{2}$, $x^{\prime}=-d-1+x.$ We define function%
\begin{align*}
f(v)  &  =\sum_{i=-d}^{d}(\left\vert v-i\right\vert +1)+\sum_{i=-(k-1)}%
^{k-1}(\left\vert v-i\right\vert +2)+\\
&  (\left\vert x^{\prime}-v\right\vert +2+\left\vert -x^{\prime}-v\right\vert
+2)
\end{align*}
Now, the definition of $G_{1}(n,d,x,s)$ implies%
\[
W(G_{1}(n,d,x,s))=W(B_{1}(n-2,d,x))+f(s)+f(d)+d-s+2.
\]
Plugging $s$ and $d$ into the formula for $f$ and simplifying the obtained
formula yields the result.
\end{proof}

As a direct consequence of Lemma \ref{Lema_WienerGr1} we obtain the following corollary.

\begin{corollary}
\label{Kor_Gr1step}It holds that%
\begin{align*}
W(G_{1}(n,d,x,1))  &  =W(G_{1}(n,d,x,0))+1,\\
W(G_{1}(n,d,x,2))  &  =W(G_{1}(n,d,x,0))+6,\\
W(G_{1}(n,d,x,-1))  &  =W(G_{1}(n,d,x,0))+3.
\end{align*}

\end{corollary}

The main tool in proving the results will be Transformation A of the graph,
which, for a given graph, finds another graph whose value of Wiener index is
greater by $4.$ Therefore, it is critical to find a graph on which
Transformation A can be applied consecutively as many times as possible. That
was the basic idea behind constructing graph $G_{1}(n,d,x,s)$ as we did, so
that we can use Theorem \ref{Tm_CaterpillarWienerComplete} in filling the
interval between values $W(G_{1}(n,d,x,s))$ for cnsecutive values of $x$ and
$d$. So, let us first apply Theorem \ref{Tm_CaterpillarWienerComplete} (i.e.
find the corresponding value of $k$) on the graph $G_{1}(n,d,x,s).$

\begin{lemma}
\label{Lema_Gr1WienerKcomplete}Let $n,d$, $x$ and $s$ be integers for which
$G_{1}(n,d,x,s)$ is defined. For $k=\frac{1}{2}n-d-4$ the interval
$[W(G_{1}(n,d,x,s)),W(G_{1}(n,d,x,s))+4k^{2}]$ is Wiener $4-$complete.
\end{lemma}

\begin{proof}
Let us denote $k_{1}=\frac{(n-2)-(2d+1)-1}{2}.$ Note that $k_{1}$ is the half
of the number of leafs appended to the vertices $u_{\pm j}$ of the interior
path of $G_{1}(n,d,x,s)$ for $j=0,\ldots,k-1.$ Since $s\leq2,$ note that the
definition of $G_{1}(n,d,x,s)$ and Theorem \ref{Tm_CaterpillarWienerComplete}
imply the result for $k=k_{1}-2.$
\end{proof}

So, let us now establish for which values of $d$ the gap between
$W(G_{1}(n,d,x,s))$ and $W(G_{1}(n,d,x-1,s))$ is smaller than $4k^{2}$ which
is a width of interval which can be filled by repeatedly applying
Transformation A on $G_{1}(n,d,x,s)$ (i.e. by using Lemma
\ref{Lema_Gr1WienerKcomplete}).

\begin{lemma}
\label{Lema_WienerGr1stepX}Let $n,d$, $x\geq2$ and $s$ be integers for which
$G_{1}(n,d,x,s)$ is defined. For $d\leq\frac{1}{2}(n-\sqrt{2n-8}-8)$ the
interval%
\[
\lbrack W(G_{1}(n,d,x,s)),W(G_{1}(n,d,x-1,s))]
\]
is Wiener $4-$complete.
\end{lemma}

\begin{proof}
First note that
\begin{align*}
W(G_{1}(n,d,x-1,s))-W(G_{1}(n,d,x,s)) &  \leq W(G_{1}(n,d,2,s))-W(G_{1}%
(n,d,1,s))=\\
&  =2(n-5)+2.
\end{align*}
Therefore, Lemma \ref{Lema_Gr1WienerKcomplete} implies it is sufficient to
find integers $n$ and $d$ for which it holds that $4k^{2}\geq2(n-5)+2$ where
$k=\frac{1}{2}n-d-4.$ By simple calculation it is easy to establish that the
inequality holds for $d\leq\frac{1}{2}(n-\sqrt{2n-8}-8)$ so the lema is proved.
\end{proof}

It is easy to show, using Lemma \ref{Lema_WienerGr1}, that $W(G_{1}%
(n,d,x-1,s))-W(G_{1}(n,d,x,s))=2n-4x$ which is divisible by $4$ since $n$ is
even. Therefore, Lemma \ref{Lema_WienerGr1stepX} enables us to "glue" together
Wiener $4-$complete intervals
\[
\lbrack W(G_{1}(n,d,x,s)),W(G_{1}(n,d,x-1,s))]
\]
into one bigger Wiener $4-$complete interval
\[
\lbrack W(G_{1}(n,d,x_{1}^{\max},s)),W(G_{1}(n,d,1,s))]
\]
where $x_{1}^{\max}=\frac{4+4d-(n-2)}{2}.$ Corrolary \ref{Kor_Gr1step} then
implies that roughly the same interval will be Wiener complete when we take
values for every $s=-1,0,1,2$. We say "roughly" because the difference
$W(G_{1}(n,d,x,2))=W(G_{1}(n,d,x,0))+6$ makes one point gap at $W(G_{1}%
(n,d,x_{1}^{\max},0))+2.$ We now want to "glue" together such bigger intervals
into one interval on the border between $d$ and $d-1.$ The problem is that
\[
G_{1}(n,d,x_{1}^{\max},s)\not =G_{1}(n,d-1,1,s),
\]
so we have to cover the gap in between. Moreover, it holds that
\[
W(G_{1}(n,d,x_{1}^{\max},s))-W(G_{1}(n,d-1,1,s))=n-3
\]
which is not divisible by 4. Therefore, we have to find enough graphs whose
values of Wiener index will cover the gap of $n-3$ plus the gap of $6$ which
arises from "rough" edge of the interval for a given $d.$

\begin{lemma}
\label{Lema_WienerGr1stepD}Let $n,d$, $x_{1}^{\max}=\frac{4+4d-(n-2)}{2}$ and
$s$ be integers for which $G_{1}(n,d,x_{1}^{\max},s)$ and $G_{1}(n,d-1,1,s)$
are defined. For $d\leq\frac{1}{2}(n-\sqrt{n+3}-6)$ the interval
\[
\lbrack G_{1}(n,d-1,1,s),G_{1}(n,d,x_{1}^{\max},s)+6]
\]
is Wiener $4-$complete.
\end{lemma}

\begin{proof}
Since
\[
W(G_{1}(n,d,x_{1}^{\max},s)+6-W(G_{1}(n,d-1,1,s))=n-3+6=n+3,
\]
Lemma \ref{Lema_Gr1WienerKcomplete} implies that it is sufficient to find for
which $d$ it holds that $4k^{2}\leq n+3$ where $k=\frac{1}{2}n-(d-1)-4.$ By
simple calculation one obtains that inequality holds for $d\geq\frac{1}%
{2}(n-\sqrt{n+3}-6)$ which proves the theorem.
\end{proof}

Note that the restriction on the maximum value of $d$ is stricter in Lemma
\ref{Lema_WienerGr1stepX} then in Lemma \ref{Lema_WienerGr1stepD} for every
$n>4.$

\bigskip

Now we have taken out all we could from graph $G_{1},$ but that covers only
caterpillars with relativly large $d.$ We can further expand Wiener complete
interval to the left side, i.e. to the caterpillars with smaller $d,$ using
graph $G_{2}$ which we will construct from the basic graph $B_{2}.$

\begin{definition}
Let $n,d$ and $x$ be integers for which $B_{2}(n-2,d,x)$ is defined. For
$s=-1,0,1,2$ caterpillar $G_{2}(n,d,x,s)$ is a caterpillar on even number of
vertices $n,$ obtained from $B_{2}(n-2,d,x)$ by appending a leaf to vertex
$u_{s}$ and a leaf to vertex $u_{d}$ of path $P=u_{-d}\ldots u_{-1}u_{0}%
u_{1}\ldots u_{d}$ in $B_{2}(n-2,d,x).$
\end{definition}

\begin{lemma}
\label{Lema_WienerGr2}Let $n,d$, $x$ and $s$ be integers for which
$G_{2}(n,d,x,s)$ is defined. Then%
\[
W(G_{2}(n,d,x,s))=W(B_{2}(n-2,d,x))+(2d+4)n-2d^{2}-7d-6-2x+2s^{2}-s.
\]

\end{lemma}

\begin{proof}
Let $k=d-1$ and $r=\frac{n-4d-2x}{2}$. We define function%
\begin{align*}
f(v) &  =\sum_{i=-d}^{d}(\left\vert v-i\right\vert +1)+\sum_{i=-(k-1)}%
^{k-1}(\left\vert v-i\right\vert +2)+\\
&  +x(\left\vert v+(d-1)\right\vert +2)+x(\left\vert v-(d-1)\right\vert +2)+\\
&  +r(\left\vert v+d\right\vert +2)+r(\left\vert v-d\right\vert +2).
\end{align*}
Now, the definition of $G_{2}(n,d,x,s)$ implies%
\[
W(G_{2}(n,d,x,s))=W(B_{2}(n-2,d,x))+f(s)+f(d)+d-s+2.
\]
Plugging $s$ and $d$ into the formula for $f$ and simplifying the obtained
formula yields the result.
\end{proof}

Again, as a direct consequence of Lemma \ref{Lema_WienerGr2} we obtain the
following corollary.

\begin{corollary}
\label{Kor_Gr2Step}It holds that%
\begin{align*}
W(G_{2}(n,d,x,1))  &  =W(G_{2}(n,d,x,0))+1,\\
W(G_{2}(n,d,x,2))  &  =W(G_{2}(n,d,x,0))+6,\\
W(G_{2}(n,d,x,-1))  &  =W(G_{2}(n,d,x,0))+3.
\end{align*}

\end{corollary}

As in the case of large $d,$ the main tool in obtaining the results will be
the following lemma.

\begin{lemma}
\label{Lema_Gr2WienerKcomplete}Let $n,d$, $x$ and $s$ be integers for which
$G_{2}(n,d,x,s)$ is defined. For $k=d-3$ the interval $[W(G_{2}%
(n,d,x,s)),W(G_{2}(n,d,x,s))+4k^{2}]$ is Wiener $4-$complete.
\end{lemma}

\begin{proof}
Let us denote $k_{1}=d-1.$ Note that $k_{1}$ is the half of the number of
leafs appended to the vertices $u_{\pm j}$ of the interior path of
$G_{2}(n,d,x,s)$ for $j=0,\ldots,k-1.$ Since $s\leq2,$ note that the
definition of $G=G_{2}(n,d,x,s)$ and Theorem
\ref{Tm_CaterpillarWienerComplete} imply the result for $k=k_{1}-2.$
\end{proof}

We first use Lemma \ref{Lema_Gr2WienerKcomplete} to cover interval between
$W(G_{2}(n,d,x,s))$ and $W(G_{2}(n,d,x-1,s)).$

\begin{lemma}
\label{Lema_WienerGr2stepX}Let $n,d$, $x\geq2$ and $s$ be integers for which
$G_{2}(n,d,x,s)$ is defined. For $d\geq\frac{1}{2}(\sqrt{2n-8}+6)$ the
interval%
\[
\lbrack W(G_{2}(n,d,x,s)),W(G_{2}(n,d,x-1,s))]
\]
is Wiener $4-$complete.
\end{lemma}

\begin{proof}
First note that for $x_{2}^{\max}=\frac{(n-2)-4d+2}{2}$ it holds that%
\begin{align*}
&  W(G_{2}(n,d,x-1,s))-W(G_{2}(n,d,x,s))\overset{}{\leq}\\
&  \overset{}{\leq}W(G_{2}(n,d,x_{2}^{\max}-1,s))-W(G_{2}(n,d,x_{2}^{\max
},s))\overset{}{=}\\
&  \overset{}{=}2(n-5)+2.
\end{align*}
Therefore, Lemma \ref{Lema_Gr2WienerKcomplete} implies it is sufficient to
find for which $n$ and $d$ it holds that $4k^{2}\geq2(n-5)+2$ where $k=d-3.$
By simple calculation it is easy to establish that the inequality holds for
$d\geq\frac{1}{2}(\sqrt{2n-8}+6)$ so the theorem is proved.
\end{proof}

Again, it is easy to show that $W(G_{2}(n,d,x-1,s))-W(G_{2}(n,d,x,s))=4d+4x-8
$ which is divisible by $4.$ Therefore, using Lema \ref{Lema_WienerGr2stepX}
we can again "glue" the interval for different values of $x$ into one bigger
interval which will be "roughly" Wiener complete when taking values of
$W(G_{2}(n,d,x,s))$ for every $s=-1,0,1,2.$ The next thing is to cover the gap
between $W(G_{2}(n,d-1,1,s))$ and $W(G_{2}(n,d,x_{2}^{\max},s))$ which equals
$n-3$ plus the gap of $6$ which arises from "rough" ends of Wiener complete
interval for given $n$ and $d.$

\begin{lemma}
\label{Lema_WienerGr2stepD}Let $n,d$, $x_{2}^{\max}=\frac{(n-2)-4d+2}{2}$ and
$s$ be integers for which $G_{2}(n,d,x_{2}^{\max},s)$ and $G_{2}(n,d-1,1,s)$
is defined. For $d\geq\frac{1}{2}(\sqrt{n+3}+8)$ the interval
\[
\lbrack W(G_{2}(n,d-1,1,s)),W(G_{2}(n,d,x_{2}^{\max},s))+6]
\]
is Wiener $4-$complete.
\end{lemma}

\begin{proof}
Since
\[
W(G_{2}(n,d,x_{2}^{\max},s))+6-W(G_{2}(n,d-1,1,s))=n+3,
\]
Lemma \ref{Lema_Gr2WienerKcomplete} implies it is sufficient to find $n$ and
$d$ for which it holds that $4k^{2}\leq n-3$ where $k=(d-1)-3.$ By simple
calculation one obtains that inequality holds for $d\geq\frac{1}{2}(\sqrt
{n+3}+8)$ which proves the theorem.
\end{proof}

\bigskip

Therefore, using graphs $G_{1}(n,d,x,s)$ and $G_{2}(n,d,x,s)$ we obtained two
big Wiener complete intervals, which it would be nice if we could "glue"
together into one big Wiener complete interval. In order to do that, note that
the equality (\ref{For_Lijepljenje}) implies%
\[
G_{2}(n,d_{2}^{\max},1,s)=G_{1}(n,d_{1}^{\min},x_{1}^{\max},s)
\]
for $d_{2}^{\max}=\left\lfloor \frac{n-2}{4}\right\rfloor $, $d_{1}^{\min
}=\left\lceil \frac{n-4}{4}\right\rceil $ and $x_{1}^{\max}=\frac
{4+4d_{1}^{\min}-(n-2)}{2}.$ Now we can state the theorem which is our main result.

\begin{theorem}
\label{Tm_MAINevenN}Let $n\geq30$, $d_{2}^{\min}=\left\lceil \frac{1}{2}%
(\sqrt{2n-8}+6)\right\rceil ,$ $x_{2}^{\max}=\frac{(n-2)-4d_{2}^{\min}+2}{2}$
and $d_{1}^{\max}=\left\lfloor \frac{1}{2}(n-\sqrt{2n-8}-8)\right\rfloor .$
The interval%
\[
\lbrack W(G_{2}(n,d_{2}^{\min},x_{2}^{\max},-1)),W(G_{1}(n,d_{1}^{\max},1,1))]
\]
is Wiener complete.
\end{theorem}

\begin{corollary}
For even $n\geq30$ it holds that $\left\vert W[\mathcal{T}_{n}]\right\vert
\geq\left\vert W^{int}[\mathcal{T}_{n}]\right\vert \geq\frac{1}{6}n^{3}%
-\frac{1}{\sqrt{2}}n^{5/2}+O(n^{2}).$
\end{corollary}

\begin{proof}
Using Theorem \ref{Tm_MAINevenN} and Lemmas \ref{Lema_WienerGr1} and
\ref{Lema_WienerGr2} it is easy to calculate that
\begin{align*}
\left\vert W^{int}[\mathcal{T}_{n}]\right\vert  &  \geq W(G_{1}(n,d_{1}^{\max
},1,1))-W(G_{2}(n,d_{2}^{\min},x_{2}^{\max},-1))=\\
&  =\frac{n^{3}}{6}-\frac{\sqrt{n^{5}-n^{4}}}{\sqrt{2}}-3n^{2}+\frac{10}%
{3}\sqrt{2n^{3}-8n^{2}}+\frac{143n}{6}+\\
&  +25\sqrt{2n-8}-25.
\end{align*}

\end{proof}

We can now prove Conjectures \ref{Con1} and \ref{Con2} we stated in the
introduction. Namely, since $\left\vert W^{int}[\mathcal{T}_{n}]\right\vert
\leq\left\vert W[\mathcal{T}_{n}]\right\vert \leq W(P_{n})-W(S_{n})=\frac
{1}{6}n^{3}-n^{2}+\frac{11}{6}n-1$, then the following holds.

\begin{theorem}
For even $n\geq30$ it holds that $\left\vert W^{int}[\mathcal{T}%
_{n}]\right\vert =\left\vert W[\mathcal{T}_{n}]\right\vert =\frac{1}{6}%
n^{3}+O(n^{2})$.
\end{theorem}

\section{Odd number of vertices}

\begin{definition}
Let $n,d$ and $x$ be integers for which $B_{1}(n-1,d,x)$ is defined. For
$s=0,1$ caterpillar $G_{3}(n,d,x,s)$ is a caterpillar on odd number of
vertices $n,$ obtained from $B_{1}(n-1,d,x)$ by appending a leaf to vertex
$u_{s}$ of path $P=u_{-d}\ldots u_{-1}u_{0}u_{1}\ldots u_{d}$ in
$B_{1}(n-1,d,x).$
\end{definition}

\begin{lemma}
\label{Lema_WienerGr3}Let $n,d$, $x$ and $s$ be integers for which
$G_{3}(n,d,x,s)$ is defined. Then%
\[
W(G_{3}(n,d,x,s))=W(B_{1}(n-1,d,x))+\frac{n^{2}}{4}-dn+2d^{2}+5d+\frac{11}%
{4}-2x+2s^{2}.
\]

\end{lemma}

\begin{proof}
Let $k=\frac{(n-1)-(2d+1)-1}{2}$, $x^{\prime}=-d-1+x.$ The definition of
$G_{3}(n,d,x,s)$ implies%
\begin{align*}
W(G_{3}(n,d,x,s))  &  =W(B_{1}(n-1,d,x))+\sum_{i=-d}^{d}(\left\vert
s-i\right\vert +1)+\\
&  +\sum_{i=-(k-1)}^{k-1}(\left\vert s-i\right\vert +2)+(s-x^{\prime}+2)+\\
&  +(-x^{\prime}-s+2).
\end{align*}
Simplifying this expression yields the result.
\end{proof}

As a direct consequence of Lemma \ref{Lema_WienerGr3} we obtain the following corollary.

\begin{corollary}
\label{Kor_Gr3Step}It holds that $W(G_{3}(n,d,x,1))=W(G_{3}(n,d,x,0))+2.$
\end{corollary}

We now want to apply Theorem \ref{Tm_CaterpillarWienerComplete} on
$G_{3}(n,d,x,s),$ i.e. establish the value of $k$ in the case of this special graph.

\begin{lemma}
\label{Lema_WienerGr3Kcomplete}Let $n,d$, $x$ and $s$ be integers for which
$G_{3}(n,d,x,s)$ is defined. For $k=\frac{1}{2}n-d-\frac{5}{2}$ the interval
$[W(G_{3}(n,d,x,s)),W(G_{3}(n,d,x,s))+4k^{2}]$ is Wiener $4-$complete.
\end{lemma}

\begin{proof}
Let us denote $k_{1}=\frac{(n-1)-(2d+1)-1}{2}.$ Note that $k_{1}$ is the half
of the number of leafs appended to the vertices $u_{\pm j}$ of the interior
path of $G_{3}(n,d,x,s)$ for $j=0,\ldots,k-1.$ Since $s\leq1,$ note that the
definition of $G_{3}(n,d,x,s)$ and Theorem \ref{Tm_CaterpillarWienerComplete}
imply the result for $k=k_{1}-1.$
\end{proof}

So, let us now establish for which values of $d$ the gap between
$W(G_{3}(n,d,x,s))$ and $W(G_{3}(n,d,x-1,s))$ is smaller than $4k^{2}$ where
$k=\frac{1}{2}n-d-\frac{5}{2}$.

\begin{lemma}
\label{Lema_WienerGr3StepX}Let $n,d$, $x\geq2$ and $s$ be integers for which
$G_{3}(n,d,x,s)$ is defined. For $d\leq\frac{1}{2}(n-\sqrt{2n-6}-5)$ the
interval%
\[
\lbrack W(G_{3}(n,d,x,s)),W(G_{3}(n,d,x-1,s))]
\]
is Wiener $4-$complete.
\end{lemma}

\begin{proof}
First note that
\begin{align*}
G_{3}(n,d,x-1,s)-G_{3}(n,d,x,s)  &  \leq G_{3}(n,d,2,s)-G_{3}(n,d,1,s)=\\
&  =2(n-4)+2.
\end{align*}
Therefore, Lemma \ref{Lema_WienerGr3Kcomplete} implies it is sufficient to
find integers $n$ and $d$ for which it holds that $4k^{2}\geq2(n-4)+2$ where
$k=\frac{1}{2}n-d-\frac{5}{2}.$ By simple calculation it is easy to establish
that the inequality holds for $d\leq\frac{1}{2}(n-\sqrt{2n-6}-5)$ so the lema
is proved.
\end{proof}

Using Lemma \ref{Lema_WienerGr3} it is easy to establish that
\[
W(G_{3}(n,d,x-1,s))-W(G_{3}(n,d,x,s))=2(n-2x+1)
\]
which is divisible by $4$ since $n$ is odd. Moreover, note that for
$x_{3}^{\max}=\frac{4+4d-(n-1)}{2}$ it holds that
\[
G_{3}(n,d,x_{3}^{\max},s)=G_{3}(n,d-1,1,s).
\]
Therefore we can use Lemma \ref{Lema_WienerGr3StepX} and "glue" together
intervals both on the border between $x$ and $x-1$ and on the border of $d$
and $d-1,$ so we will obtain one large interval which is Wiener $2-$complete
(because of Corollary \ref{Kor_Gr3Step}).

\bigskip

Again, here we have used $G_{3}(n,d,x,s)$ to the maximum, but we have covered
thus only caterpillars with large $d.$ Let us now use graph $B_{2}(n,d,x)$ to
create fourth special kind of caterpillars which we will use to widen our
interval to caterpillars with small $d.$

\begin{definition}
Let $n,d$ and $x$ be integers for which $B_{2}(n-1,d,x)$ is defined. For
$s=0,1$ caterpillar $G_{4}(n,d,x,s)$ is a caterpillar on odd number of
vertices $n,$ obtained from $B_{2}(n-1,d,x)$ by appending a leaf to vertex
$u_{s}$ of path $P=u_{-d}\ldots u_{-1}u_{0}u_{1}\ldots u_{d}$ in
$B_{2}(n-1,d,x).$
\end{definition}

\begin{lemma}
\label{Lema_WienerGr4}Let $n,d$, $x$ and $s$ be integers for which
$G_{4}(n,d,x,s)$ is defined. Then%
\[
W(G_{4}(n,d,x,s))=W(B_{2}(n-1,d,x))+(2+d)n-2d^{2}-3d-1+2s^{2}-2x.
\]

\end{lemma}

\begin{proof}
Let $k=d-1$ and $r=\frac{(n-1)-4d-2x+2}{2}.$ The definition of $G_{4}%
(n,d,x,s)$ implies%
\begin{align*}
W(G_{4}(n,d,x,s))  &  =W(B_{2}(n-1,d,x))+\sum_{i=-d}^{d}(\left\vert
s-i\right\vert +1)+\\
&  +\sum_{i=-(k-1)}^{k-1}(\left\vert s-i\right\vert +2)+(s-x^{\prime}+2)+\\
&  +2x(d+1)+2r(d+2).
\end{align*}
Simplifying this expression yields the result.
\end{proof}

\begin{corollary}
\label{Kor_gr4step}It holds that $W(G_{4}(n,d,x,1))=W(G_{4}(n,d,x,0))+2.$
\end{corollary}

Let us now apply Theorem \ref{Tm_CaterpillarWienerComplete} on $G_{4}%
(n,d,x,s).$

\begin{lemma}
\label{Lema_Gr4WienerKcomplete}Let $n,d$, $x$ and $s$ be integers for which
$G_{4}(n,d,x,s)$ is defined. For $k=d-2$ the interval $[W(G_{4}%
(n,d,x,s)),W(G_{4}(n,d,x,s))+4k^{2}]$ is Wiener $4-$complete.
\end{lemma}

\begin{proof}
Let us denote $k_{1}=d-1.$ Note that $k_{1}$ is the half of the number of
leafs appended to the vertices $u_{\pm j}$ of the interior path of
$G_{3}(n,d,x,s)$ for $j=0,\ldots,k-1.$ Since $s\leq1,$ note that the
definition of $G_{4}(n,d,x,s)$ and Theorem \ref{Tm_CaterpillarWienerComplete}
imply the result for $k=k_{1}-1.$
\end{proof}

Now we can establish the minimum value of $d$ for which the difference between
Wiener index of $G_{4}(n,d,x,s)$ and $G_{4}(n,d,x-1,s)$ can be "covered" by
Transformation A.

\begin{lemma}
\label{Lema_WienerGr4StepX}Let $n,d$, $x\geq2$ and $s$ be integers for which
$G_{4}(n,d,x,s)$ is defined. For $d\geq\frac{1}{2}(\sqrt{2n-6}+4)$ the
interval%
\[
\lbrack W(G_{4}(n,d,x,s)),W(G_{4}(n,d,x-1,s))]
\]
is Wiener $4-$complete.
\end{lemma}

\begin{proof}
First note that for $x_{4}^{\max}=\frac{(n-1)-4d+2}{2}$ it holds that%
\begin{align*}
W(G_{4}(n,d,x-1,s))-W(G_{4}(n,d,x,s)) &  \leq W(G_{4}(n,d,x_{4}^{\max
}-1,s))-W(G_{4}(n,d,x_{4}^{\max},s))=\\
&  =2(n-4)+2.
\end{align*}
Therefore, Lemma \ref{Lema_Gr4WienerKcomplete} implies it is sufficient to
find integers $n$ and $d$ for which it holds that $4k^{2}\geq2(n-4)+2$ where
$k=d-2.$ By simple calculation it is easy to establish that the inequality
holds for $d\leq\frac{1}{2}(n-\sqrt{2n-6}-5)$ so the lema is proved.
\end{proof}

Using Lemma \ref{Lema_WienerGr4} it is easy to establish that
\[
W(G_{4}(n,d,x-1,s))-W(G_{4}(n,d,x,s))=4(x+2d-2)
\]
which is divisible by $4$. Moreover, note that for $x_{4}^{\max}%
=\frac{(n-1)-4d+2}{2}$ it holds that
\[
G_{4}(n,d,x_{4}^{\max},s)=G_{4}(n,d-1,1,s).
\]
Therefore we can use Lemma \ref{Lema_WienerGr4StepX} and "glue" together
intervals both on the border between $x$ and $x-1$ and on the border of $d$
and $d-1,$ so we will obtain one large interval which is Wiener $2-$complete
(because of Corollary \ref{Kor_gr4step}).

Finally, noting that for $d_{3}^{\min}=\left\lceil \frac{n-3}{4}\right\rceil
$, $x_{3}^{\max}=\frac{4+4d-(n-1)}{2}$ and $d_{4}^{\max}=\left\lfloor
\frac{n-1}{4}\right\rfloor $ it holds that%
\[
G_{3}(n,d_{3}^{\min},x_{3}^{\max},s)=G_{4}(n,d_{4}^{\max},1,s),
\]
we conclude that we can "glue" together two large Wiener $2-$complete
intervals we obtained (one for large values of $d$ and the other for small
values of $d$), and thus prove our main result.

\begin{theorem}
\label{Tm_MAINoddN}Let $n\geq21$, $d_{3}^{\max}=\frac{1}{2}(n-\sqrt{2n-6}-5)$,
$d_{4}^{\min}=\frac{1}{2}(\sqrt{2n-6}+4)$ and $x_{4}^{\max}=\frac
{(n-1)-4d_{4}^{\min}+2}{2}$ The interval%
\[
\lbrack W(G_{4}(n,d_{4}^{\min},x_{4}^{\max},0)),W(G_{3}(n,d_{3}^{\max},1,1))]
\]
is Wiener $2-$complete.
\end{theorem}

\begin{corollary}
For odd $n\geq21$ it holds that $\left\vert W[\mathcal{T}_{n}]\right\vert
\geq\left\vert W^{int}[\mathcal{T}_{n}]\right\vert \geq\frac{1}{12}n^{3}%
-\frac{1}{2\sqrt{2}}n^{5/2}+O(n^{2}).$
\end{corollary}

\begin{proof}
Using Theorem \ref{Tm_MAINevenN} and Lemmas \ref{Lema_WienerGr1} and
\ref{Lema_WienerGr2} it is easy to calculate that
\begin{align*}
\left\vert W^{int}[\mathcal{T}_{n}]\right\vert  &  \geq(W(G_{3}(n,d_{3}^{\max
},1,s))-W(G_{4}(n,d_{4}^{\min},x_{4}^{\max},s)))/2=\\
&  =\frac{n^{3}}{12}-\frac{\sqrt{n^{5}-3n^{4}}}{2\sqrt{2}}-n^{2}+\frac{5}%
{3}\sqrt{2n^{3}-6n^{2}}+\frac{83n}{12}+\frac{11\sqrt{n-3}}{6\sqrt{2}}-13.
\end{align*}

\end{proof}

We can now prove Conjectures \ref{Con1} and \ref{Con2} we stated in the
introduction (to be more precise - prove the adjusted version of conjectures).
Namely, Theorem \ref{tm_WienerEven} implies $\left\vert W^{int}[\mathcal{T}%
_{n}]\right\vert \leq\left\vert W[\mathcal{T}_{n}]\right\vert \leq
(W(P_{n})-W(S_{n}))/2=\frac{1}{12}n^{3}-\frac{1}{2}n^{2}+\frac{11}{12}%
n-\frac{1}{2}$. Therefore the following theorem is proved.

\begin{theorem}
For odd $n\geq21$ it holds that $\left\vert W^{int}[\mathcal{T}_{n}%
]\right\vert =\left\vert W[\mathcal{T}_{n}]\right\vert =\frac{1}{12}%
n^{3}+O(n^{2})$.
\end{theorem}

\section*{Acknowledgments}

This work has been supported in part by Croatian Science Foundation under the
project 8481 (BioAmpMode) and Croatian-Chinese bilateral project
\textquotedblleft Graph-theoretical methods for nanostructures and
nanomaterials\textquotedblright.


\begin{thebibliography}{9}                                                                                                %


\bibitem {BanBeregMustafa(2004)}Y. A. Ban, S. Bereg, N. H. Mustafa, On a
conjecture on Wiener indices in combinatorial chemistry, Algorithmica 40
(2004) 99--117.

\bibitem {GutmanYeh(1995)}I. Gutman, Y. N. Yeh, The sum of all distances in
bipartite graphs, Math. Slovaca 45 (1995) 327-334.

\bibitem {KnorSkrekovskiTepeh(2016)}M. Knor, R. \v{S}krekovski, A. Tepeh,
Mathematical aspects of Wiener index, Ars math. contemp. 11. (2016) 327-352.

\bibitem {KrncSkrekovski(2016)}M. Krnc and R. \v{S}krekovski, On Wiener
Inverse Interval Problem, MATCH Commun. Math. Comput. Chem. 75 (2016) 71--80.

\bibitem {LepovicGutman(1998)}M. Lepovic, I. Gutman, A collective property of
trees and chemical trees, J. Chem. Inf. Comput. Sci. 38 (1998) 823--826.

\bibitem {GutmanSurvey(2014)}K. Xu, M. Liu, K. C. Das, I. Gutman, B. Furtula,
A survey on graphs extremal with respect to distance based topological
indices, MATCH Commun. Math. Comput. Chem. 71 (2014) 461-508.

\bibitem {Wagner(2006)}S. G. Wagner, A class of trees and its Wiener index,
Acta Appl. Math. 91 (2006) 119-132.

\bibitem {Wang(2006)}H. Wang, G. Yu, All but 49 numbers are Wiener indices of
trees, Acta Appl. Math. 92 (2006) 15-20.

\bibitem {Wiener(1946)}H. Wiener, Structural determination of para n boiling
points, J. Am. Chem. Soc. 69 (1947) 17-20.
\end{thebibliography}
\end{document}